\documentclass[a4paper,12pt]{article}
\usepackage[T1]{fontenc}
\usepackage[latin1]{inputenc}

\usepackage{amssymb}
\usepackage{amsmath}
\usepackage[english]{babel}
\usepackage{graphics}
\usepackage{latexsym}
\usepackage{amsfonts}
\numberwithin{equation}{section}

\usepackage{indentfirst}
\newtheorem{thm}{Theorem}[section]
\newtheorem{lem}{Lemma}[section]
\newtheorem{coro}{Corollary}[section]
\newtheorem{prop}{Proposition}[section]
\newtheorem{defn}{Definition}[section]
\newtheorem{exple}{Examples}[section]
\newtheorem{rem}{Remark}[section]
\newenvironment{proof}[1][Proof]{\textbf{#1.} }{\hfill $\Box$}
\usepackage[colorlinks,linkcolor=red,urlcolor=blue,citecolor=blue,
pdftitle={New classes of processes in stochastic calculus for signed measures},
pdfauthor={Fulgence EYI OBIANG},
pdfsubject={Sujet du document},
pdfkeywords={Signed measures,Stochastic differential equations, relative martingales, enlargement of filtrations},
bookmarks, 
pdfstartview=FitH, 
colorlinks=true, 
bookmarksnumbered=true, 
pdfpagemode=UseOutlines 
,pdfpagelabels,breaklinks
]{hyperref}
\newcommand{\R}{\mathbb R}
\newcommand{\Qv}{\mathbb Q}
\newcommand{\Pv}{\mathbb P}
\newcommand{\E}{\mathbb E}

\newcommand{\Dv}{\mathbb D}

\begin{document}
\title{New classes of processes in stochastic calculus for signed measures\footnote{This work is supported by Hassan II Academy of Sciences and Techniques, project "Mathematics and applications"}}
\date{}
\author{Fulgence EYI OBIANG \footnote{\textbf{feyiobiang@yahoo.fr}, Université Cadi Ayyad de Marrakesh \& Université des Sciences et Techniques de Masuku.}\,,
\, Youssef OUKNINE\footnote{\textbf{ouknine@ucam.ac.ma}, Université Cadi Ayyad de Marrakesh  and Hassan II Academy of Sciences and Technologies Rabat.}\, ,
\, Octave MOUTSINGA\footnote{\textbf{octavemoutsing-pro@yahoo.fr}, Université des Sciences et Techniques de Masuku.}
}
\maketitle
\begin{abstract}
Let us consider a signed measure $\Qv$ and a probability measure $\Pv$ such that $\Qv<<\Pv$. Let $D$ be the density of $\Qv$ with respect to $\Pv$. $H$ represents the set of zeros 
of $D$, $\overline{g}=0\vee\sup{H}$. In this paper, we shall consider two classes of nonnegative processes of the form $X_{t}=N_{t}+A_{t}$. The first one is the  class of 
semimartingales where $ND$ is a cadlag local martingale and $A$ is a continuous and non-decreasing process such that $(dA_{t})$ is carried by $H\cup\{t: X_{t}=0\}$. The second one is
 the case where $N$ and $A$ are null on $H$ and $A_{.+\overline{g}}$ is a non-decreasing, continuous process such that $(dA_{t+\overline{g}})$ is carried by
 $\{t: X_{t+\overline{g}}=0\}$. We shall show that these classes are extensions of the class $(\sum)$ defined by A.Nikeghbali \cite{nik} in the framework of stochastic calculus for 
 signed measures.
\end{abstract}
\underline{\bf Keywords:}~~Signed measures, enlargement of filtration, relative martingales, class $(\sum)$.
\section*{Introduction}
P.A.Meyer suggested the study of signed measures to generalize the famous Paul Lévy's theorem (which characterizes Wiener's measure on $\Omega=\textsl{C}_{0}(\R_{+},\R)$, as the unique probability measure under which $X$ and $(X_{t}^{2}-t)_{t\geq0}$ are martingales). Such a study  was conducted by J.R.Chavez ($\cite{chav}$,1984), where a definition of a martingale with respect to a signed measure instead a probability measure is given. After him S.Beghdadi Sakrani ($\cite{sak}$, 2003) proposed a study of stochastic calculus for signed measures. From a definition of a martingale with respect to a signed measure different of the definition used by J.R.Chavez, she defines a stochastic integral with respect to a signed measure from which she extends some results of the theory of stochastic calculus. For instance, she gives an Ito formula and Girsanov theorem for signed measures.

It is well known that the processes of the following form:
$$X_{t}=N_{t}+A_{t}$$
where $A$ is a non-decreasing and continuous process such that $(dA_{t})$ is carried by \\$\{t: X_{t}=0\}$, have played a capital role in many probabilistic studies. For instance: the family of Azema-Yor martingales, the resolution of Skorokhod's embeding problem, the study of Brownian local times and the study of zeros of continuous martingales \cite{1}. In this sense A.Nikeghbali in \cite{nik} has provided a general framework and methods, based on martingale techniques, to deal with a large class of this type of processes. More precisely, this class is denoted $\big(\sum\big)$ and is defined as follows.
\begin{defn}
Let $(X_{t})$ be a nonnegative local submartingale, which decomposes as 
$$X_{t}=N_{t}+A_{t}.$$
We say that $(X_{t})$ is of class $\big(\sum\big)$ if:
\begin{enumerate}
	\item $(N_{t})$ is a cadlag local martingale, with $N_{0} = 0$;
	\item $(A_{t})$is a continuous non-decreasing process, with $A_{0} = 0$;
	\item the measure $(dA_{t})$ is carried by the set $\{t: X_{t}=0\}$.
\end{enumerate}
\end{defn}

The aim of this paper is to extend in signed measures theory some results of A.Nikeghbali \cite{nik} by defining two new classes of processes similar to the class $\big(\sum\big)$. Therefore in section 1, we  shall give some general notations used in this paper. In section 2, we shall establish the first class of processes by using uniquely the signed measures theory which has been developed by J.R.Chavez. In section 3, we will establish the second class of processes relying on the works of S.Beghdadi Sakrani.

\section{Some notations}
We start by giving some notations which will be used in this paper. Consider a measure space $(\Omega, \mathcal{F}_{\infty}, \Qv)$, where $\Qv$ is a bounded signed measure. Let $\Pv$ be a probability measure on $\mathcal{F}_{\infty}$ such that $\Qv<<\Pv$. We shall always use the following notations:
\begin{itemize}
	\item $D_{t}=\frac{d\Qv|_{\mathcal{F}_{t}}}{d\Pv|_{\mathcal{F}_{t}}}$ where $\mathcal{F}$ is a right continuous filtration completed with respect to $\Pv$ such that $\mathcal{F}_{\infty}=\vee_{t}\mathcal{F}_{t}$. We shall consider that $D$ is continuous in this paper. Note also that $D$ is a uniformly integrable martingale (see S.Beghdadi Sakrani \cite{sak}).
	\item $H=\{t: D_{t}=0\}$
	\item $g=\sup{H}$; $\overline{g}=0\vee g$; $\gamma_{t}=0\vee\sup\{s\leq t, D_{s}=0\}$; $\overline{g}_{t}=0\vee\sup\{s<t, D_{s}=0\}$.
	\item $G$ denotes the set of left endpoints of $H^{c}$.
	\item The smallest right continuous filtration containing $(\mathcal{F}_{t})$ for which $\overline{g}$ is a stopping time will be noted $(\mathcal{F}^{g}_{t})$.
 Then, the filtration $(\mathcal{F}^{g}_{\overline{g}+t})$ is well defined and will be denoted $(\mathcal{F}_{t+g})$.
	\item If $X$ is an adapted process with respect to $(\mathcal{F}_{t})$, we shall denote $\widetilde{X}:=X_{.+\overline{g}}$.
	\item We shall note $\Pv^{'}=\frac{|D_{\infty}|}{\E(|D_{\infty}|)}\Pv$.
\end{itemize}

Now, we shall introduce the first class of processes in next section.
\section{Presentation of first class.}

In this section, we will define a new class of nonnegative $\Pv$-semimartingales  which is an extension of class $(\sum)$  defined by A.Nikeghbali in \cite{nik}. Note that for reasons that we shall explain further, we shall use essentially  the tools of stochastic calculus for signed measure established by J.R.Chavez \cite{chav}. Therefore,  we shall first recall the following definition of a martingale with respect to a signed measure of J.R.Chavez in \cite{chav}.
\begin{defn}
We consider a measure space $(\Omega,\mathcal{F}_{\infty},\Qv)$, where $\Qv$ is a bounded signed measure.
Let $\Pv$ be a probability on $\mathcal{F}_{\infty}$ such that $\Qv<<\Pv$. $(\mathcal{F}_{t})_{t\geq0}$ is a right continuous filtration, completed with respect to $\Pv$ such that $\mathcal{F}_{\infty}=\vee_{t}\mathcal{F}_{t}$ and $D_{t}=\frac{d\Qv|_{\mathcal{F}_{t}}}{d\Pv|_{\mathcal{F}_{t}}}$. We say that, a $(\mathcal{F}_{t})_{t\geq0}$-adapted process $X$ is a $(\Qv,\Pv)$-martingale if:
\begin{enumerate}
	\item $X$ is a $\Pv-$semimartingale.
	\item $XD$ is a $\Pv-$martingale.
\end{enumerate}
\end{defn}

Now we define the new class of semimartingales  that we talked about above.
\begin{defn}
Let $X$ be a nonnegative $\Pv$-semimartingale, which decomposes as:
$$X_{t}=M_{t}+A_{t}.$$
We say that $X$ is of class $(\sum(H))$ if:
\begin{enumerate}
	\item $M$ is a càdlàg $(\Qv,\Pv)$-local martingale, with $M_{0}=0$;
	\item $A$ is a continuous non-decreasing process, with $A_{0}=0$;
	\item the mesure $(dA_{t})$ is carried by the set $\{t: X_{t}=0\}\cup H$.
\end{enumerate}
\end{defn}
\begin{rem}
If $\Qv$ is a  probability measure and if we take $\Pv=\Qv$, then $$\left(\sum(H)\right)=\left(\sum\right).$$
\end{rem}

\begin{coro}\label{c0}
Let $X=N+A$ be a process of the class $\left(\sum(H)\right)$ such that \\$H\subset \{t: X_{t}=0\}$ and $A_{\overline{g}}=0$. Then, the process $\widetilde{X}$ is of class $\big(\sum\big)$.
\end{coro}
\begin{proof}
We have
$$\widetilde{X}=\widetilde{N} +\widetilde{A}.$$
By assumptions  $\widetilde{A}$ is continuous, non-decreasing, vanishes at zero and  $d(\widetilde{A}_{t})$ is carried by $\{t: \widetilde{X}_{t}=0\}$. Furthermore, 
$DN$ is a $\Pv-$local martingale since $N$ is a $(\Qv,\Pv)-$local martingale. From the theorem (4.2.1) of J.Azema and M.Yor \cite{1}, $\widetilde{N}$ is a $\Pv^{'}-$local martingale. Furthermore, $\widetilde{N}_{0}=N_{\overline{g}}=0$. Indeed $\forall t\leq\overline{g}$, $A_{t}=0$ since $A$ is a nonnegative and non-decreasing process with $A_{\overline{g}}=0$. Then, $N$ is null on $H$ because $H\subset \{t: X_{t}=0\}$. Hence, $N_{\overline{g}}=0$ and  $\widetilde{X}$ is of class $\big(\sum\big)$.
\end{proof}

We have the following martingale characterization for the processes of class $\left(\sum(H)\right)$.
\begin{thm}\label{t1}
The following are equivalent:
\begin{enumerate}
	\item the process $X$ is of class $\left(\sum(H)\right)$;
	\item there exists a non-decreasing, adapted and continuous process $C$ such that for all locally bounded Borel function $f$, and $F(x)=\int_{0}^{x}{f(z)dz}$, the process
	$$F(C_{t})-f(C_{t})X_{t}$$
is a $(\Qv,\Pv)$-local martingale.
\end{enumerate}
\end{thm}
\begin{proof}
 $(1)\Rightarrow(2)$\\
 First, let us assume that $f$ is $\mathcal{C}^{1}$ and let us take $C_{t}=A_{t}$. An integration by parts give:
 $$f(A_{t})X_{t}=\int_{0}^{t}{f(A_{u})dX_{u}}+\int_{0}^{t}{f^{'}(A_{u})X_{u}dA_{u}}$$
 $$\hspace{4.5cm}=\int_{0}^{t}{f(A_{u})dM_{u}}+\int_{0}^{t}{f(A_{u})dA_{u}}+\int_{0}^{t}{f^{'}(A_{u})X_{u}dA_{u}}.$$
 Since $F(A_{t})=\int_{0}^{t}{f(A_{u})dA_{u}}$, we get:
 $$f(A_{t})X_{t}-F(A_{t})=\int_{0}^{t}{f(A_{u})dM_{u}}+\int_{0}^{t}{f^{'}(A_{u})X_{u}dA_{u}}.$$
Hence,
$$D_{t}(f(A_{t})X_{t}-F(A_{t}))=D_{t}\left(\int_{0}^{t}{f(A_{u})dM_{u}}\right)+D_{t}\left(\int_{0}^{t}{f^{'}(A_{u})X_{u}dA_{u}}\right).$$
But according to J.R.Chavez proposition 2 \cite{chav}, $\int_{0}^{t}{f(A_{u})dM_{u}}$ is a $(\Qv,\Pv)-$local martingale. Then,
$$D_{t}\left(\int_{0}^{t}{f(A_{u})dM_{u}}\right)$$ 
is a $\Pv-$local martingale. Moreover taking $Y_{t}=\int_{0}^{t}{f^{'}(A_{u})X_{u}dA_{u}}$, we obtain after an integration by parts:
$$D_{t}Y_{t}=\int_{0}^{t}{D_{s}dY_{s}}+\int_{0}^{t}{Y_{s}dD_{s}}$$
$$\hspace{2.5cm}=\int_{0}^{t}{D_{s}X_{s}f^{'}(A_{s})dA_{s}}+\int_{0}^{t}{Y_{s}dD_{s}}$$
Since $(dA_{t})$ is carried by the set $\{t: X_{t}=0\}\cup H$, we have $\int_{0}^{t}{f^{'}(A_{u})D_{u}X_{u}dA_{u}}=0$.
Thus, $(D_{t}Y_{t})_{t\geq0}$ is a $\Pv-$local martingale. Therefore, $(D_{t}(f(A_{t})X_{t}-F(A_{t})))_{t\geq0}$ is a $\Pv-$local martingale.
 Consequently, $(F(A_{t})-f(A_{t})X_{t})_{t\geq0}$ is a $(\Qv,\Pv)-$local martingale. The general case when f is only assumed to be locally bounded follows from a monotone class argument and the integral
representation is still valid.\\
 $(2)\Rightarrow(1)$\\
 First take $F (a)=a$; we then obtain that $C_{t}-X_{t}$ is a $(\Qv,\Pv)$-local martingale. Next, we take $F(a)=a^{2}$ and we get $C^{2}_{t}-2C_{t}X_{t}$ is a   $(\Qv,\Pv)$-local martingale. From Itô formula, we get:
 $$C_{t}^{2}-2C_{t}X_{t}=2\int_{0}^{t}{C_{s}d(C_{s}-X_{s})}-2\int_{0}^{t}{X_{s}dC_{s}}$$
 Hence, we must have $Y_{t}=\int_{0}^{t}{X_{s}dC_{s}}$ is a $(\Qv,\Pv)$-local martingale. Hence, $DY$
 is a $\Pv$-local martingale. An integration by parts give:
 $$D_{t}Y_{t}=\int_{0}^{t}{Y_{s}dD_{s}}+\int_{0}^{t}{D_{s}dY_{s}}$$
 $$D_{t}Y_{t}=\int_{0}^{t}{Y_{s}dD_{s}}+\int_{0}^{t}{D_{s}X_{s}dC_{s}}$$
 Since $\int_{0}^{t}{Y_{s}dD_{s}}$ is a $\Pv$-local martingale, we get:
 $$\int_{0}^{t}{D_{s}X_{s}dC_{s}}=0.$$
 Then, $(dC_{t})$ is carried by $\{t: X_{t}=0\}\cup H$.\\
 Therefore, $X$ is of class $\left(\sum(H)\right)$.
 \end{proof}
 
 One often needs to know when $(F(A_{t})-f(A_{t})X_{t})_{t\geq0}$ is a true martingale. The following corollary gives us an answers.
\begin{coro}\label{r1}
 Let $X$ be of class $\left(\sum(H)\right)$ and of class $\Dv$. If $f$ is a Borel bounded function with compact support, then $(F(A_{t})-f(A_{t})X_{t})_{t\geq0}$ is a uniformly integrable $(\Qv,\Pv)$- martingale. Recall that a stochastic process $X$ is said of class $(\Dv)$ if \\$\{X_{\tau}:$ $\tau<\infty$ $is$ $a$ $stopping$ $time\}$ is uniformly integrable.
\end{coro}
\begin{proof}
There exist two constants $C>0$, $K>0$ such that $\forall x\geq0$, $|f(x)|\leq C$ , and $\forall x\geq K$, $f(x)=0$. Moreover we have:
$$|D_{t}(F(A_{t})-f(A_{t})X_{t})|=|D_{t}||(F(A_{t})-f(A_{t})X_{t})|$$
$$\hspace{4cm}\leq R|(F(A_{t})-f(A_{t})X_{t})|$$
Because $D$ is bounded. Consequently, we have
$$|D_{t}(F(A_{t})-f(A_{t})X_{t})|\leq RCK+RCX_{t};$$
now, since $(RCK+RCX_{t})_{t\geq0}$ is of class $(\Dv)$, we deduce that \\$\left(D_{t}(F(A_{t})-f(A_{t})X_{t})\right)_{t\geq0}$ is a $\Pv$-local
martingale of class $(\Dv)$ and hence it is a uniformly integrable martingale.\\
Therefore, $(F(A_{t})-f(A_{t})X_{t})_{t\geq0}$ is a uniformly integrable $(\Qv,\Pv)$-martingale.
\end{proof}

 Now, we give two properties of class $\left(\sum(H)\right)$ in next proposition.
\begin{prop}
\begin{enumerate}
	\item If $f$ is a nonenegative and locally bounded Borel function, then $\forall$ $X$ $\in$ $\left(\sum(H)\right)$, $(f(A_{t})X_{t})_{t\geq0}$ is also of class $\left(\sum(H)\right)$ and its non-decreasing part is $(F(A_{t}))_{t\geq0}$.
	\item Let $(X_{t}^{1}),\ldots,(X_{t}^{n})$ be processes of class $\left(\sum(H)\right)$ such that \\
$\langle X^{i},X^{j} \rangle_{t}=0$ for $i\neq j.$ Then $\big(\Pi^{n}_{i=1}X^{i}_{t}\big)_{t\geq0}$ is also of class $\left(\sum(H)\right)$.
\end{enumerate}
\end{prop}
\begin{proof}
\begin{enumerate}
	\item According to theorem \ref{t1}, $(f(A_{t})X_{t}-F(A_{t}))_{t\geq0}$ is a $(\Qv,\Pv)-$ local martingale which vanishes at zero. Moreover, $(F(A_{t}))_{t\geq0}$ is a continuous,   non-decreasing process which vanishes at zero since $f$ is a nonnegative function. Then, $(f(A_{t})X_{t})_{t\geq0}$ is of class $\left(\sum(H)\right)$ beacause, it is easy to see that $(dF(A_{t}))$ is carried by $H\cup\{t: f(A_{t})X_{t}=0\}$.
	\item Since $[X^{1},X^{2}]_{t}=0$, integration by parts yields:	$$X^{1}_{t}X^{2}_{t}=\int_{0}^{t}{X^{1}_{u^{-}}dM_{u}^{2}}+\int_{0}^{t}{X^{2}_{u^{-}}dM_{u}^{1}}+\int_{0}^{t}{X^{1}_{u}dA_{u}^{2}}+\int_{0}^{t}{X^{2}_{u}dA_{u}^{1}}.$$
	$\int_{0}^{t}{X^{1}_{u^{-}}dM_{u}^{2}}+\int_{0}^{t}{X^{2}_{u^{-}}dM_{u}^{1}}$ is a $(\Qv,\Pv)-$ local martingale which vanishes at zero. And \\
	$A_{t}=\int_{0}^{t}{X^{1}_{u}dA_{u}^{2}}+\int_{0}^{t}{X^{2}_{u}dA_{u}^{1}}$ is a continuous and non-decreasing process which vanishes at zero. Moreover,
	$$dA_{t}=X^{1}_{t}dA^{2}_{t}+X^{2}_{t}dA^{1}_{t}$$
	is carried by $H\cup\{t: X^{1}_{t}X^{2}_{t}=0\}$. Therefore, $X^{1}X^{2}$ is of class $\left(\sum(H)\right)$. If $n\geq3$, then $[X^{1}X^{2},X^{3}]_{t}=0$, and it follows by induction.
\end{enumerate}
\end{proof}

Now, we shall give some examples of processes in the class  $\left(\sum(H)\right)$.
\begin{exple}
\begin{enumerate}
	\item Let $M$ be a continuous $(\Qv,\Pv)-$ local martingale with respect to some filtration $(\mathcal{F}_{t})_{t\geq0}$, starting from 0; then,
	$$|M_{t}|=\int_{0}^{t}sgn(M_{s})dM_{s}+L_{t}^{0}(M)$$
	is of class $\left(\sum(H)\right)$.
	\item For any $\alpha>0$, $\beta>0$, the process:
	$$\alpha M_{t}^{+}+\beta M_{t}^{-}$$
	is also of class $\left(\sum(H)\right)$.
	\item Let $M$ be a $(\Qv,\Pv)-$ local martingale which vanishes at zero with only negative jumps and let $S$ its supremum process; then
	$$X_{t}\equiv S_{t}-M_{t}$$
	is of class $\left(\sum(H)\right)$.
\end{enumerate}
\end{exple}

\begin{rem}\label{mart}
If $X$ is a cadlag, uniformly integrable $(\Qv,\Pv)$-martingale (respectively,a $(\Qv,\Pv)-$local martingale) with respect to the filtration $(\mathcal{F}_{t})_{t\geq0}$. Then, $(\widetilde{X}_{t})_{t\geq0}$ is a uniformly integrable $\Pv^{'}-$martingale (respectively,a $\Pv^{'}-$local martingale) with respect to the filtration $(\mathcal{F}_{t+g})_{t\geq0}$.
\end{rem}

Indeed by assumptions, $XD$ is a right continuous, uniformly integrable \\$((\mathcal{F}_{t})_{t\geq0},\Pv)-$ martingale (respectively, local martingale) null on $H$. Then, applying the quotient theorem of J.Azema and M.Yor \cite{1}, we obtain that: $(sgn(D_{t+\overline{g}})\widetilde{X}_{t})_{t\geq0}$ is a $\Pv^{'}-$martingale with respect to the filtration $(\mathcal{F}_{t+g})_{t\geq0}$. Since $D$ is a continuous process, hence $(sgn(D_{t+\overline{g}}))_{t\geq0}$ is a constant process. Concequently, $(\widetilde{X}_{t})_{t\geq0}$ is a $\Pv^{'}-$martingale with respect to the filtration $(\mathcal{F}_{t+g})_{t\geq0}$.

In next lemma, we should extend  the Doob's Maximal Identity obtained by A.Nikeghbali and M.Yor \cite{max} to  the case of $(\Qv,\Pv)$-martingales.
\begin{lem}\label{l1}
Let $X$ be a nonnegative, $(\Qv,\Pv)-$local martingale such that:
$$\lim_{t\rightarrow\infty}{X_{t}}=0.$$
Let us note: $\widetilde{S}_{t}=\sup_{u\leq t}{\widetilde{X}_{u}}$ and $\widetilde{S}_{t}^{T}=\sup_{u\leq t}{\widetilde{X}_{u\wedge T}}$. If $\widetilde{S}$ is continuous, then for any $a>0$, we have
\begin{enumerate}
	\item $\Pv^{'}(\widetilde{S}_{\infty}>a|\mathcal{F}_{\overline{g}}^{g})=\left(\frac{X_{\overline{g}}}{a}\right)\wedge1$
	\item For every stopping time $T$, $\Pv^{'}(\widetilde{S}^{T}_{\infty}>a|\mathcal{F}_{g+T})=\left(\frac{X_{\overline{g}+T}}{a}\right)\wedge1$
\end{enumerate}
\end{lem}
\begin{proof}
From the remark \ref{mart}, $\widetilde{X}$ is a $\Pv^{'}-$local martingale with respect to the filtration $(\mathcal{F}_{t+g})_{t\geq0}$. Moreover we have:
$$\lim_{t\rightarrow\infty}{\widetilde{X}_{t}}=\lim_{t\rightarrow\infty}{X_{t}}=0.$$
Then applying the Doob's Maximal Identity of A.Nikeghbali and M.Yor \cite{max}, we get:
\begin{enumerate}
	\item $\Pv^{'}(\widetilde{S}_{\infty}>a|\mathcal{F}_{\overline{g}}^{g})=\left(\frac{X_{\overline{g}}}{a}\right)\wedge1$
	\item For every stopping time T, $\Pv^{'}(\widetilde{S}^{T}_{\infty}>a|\mathcal{F}_{g+T})=\left(\frac{X_{\overline{g}+T}}{a}\right)\wedge1$
\end{enumerate}
\end{proof}

\begin{rem}
If $\Qv$ is a probability measure, and if we take $\Pv=\Qv$, then we shall obtain the Doob's Maximal Identity obtained by A.Nikeghbali and M.Yor \cite{max}.
\end{rem}

Now, we shall extend one of main resuts of A.Nikeghbali \cite{nik}. But first, do the following notations:
$F(x)=1-exp\left(-\int_{x}^{+\infty}\frac{du}{\varphi(u)}\right)$, $f(x)=\frac{-1}{\varphi(x)}(1-F(x))$. We remark that $f$ is the derivative function of $F$. We also note: $M_{t}=F(A_{t})-f(A_{t})X_{t}$ and $M^{u}_{t}=F_{u}(A_{t})-f_{u}(A_{t})X_{t}$  with $F_{u}(x)=1-exp\left(-\int_{x}^{u}\frac{dz}{\varphi(z)}\right)$, $x<u$.

\begin{thm}
Let $X$ be a process of class $\left(\sum(H)\right)$, with only negative jumps such that $A_{\infty}=\infty$. Define $(\tau_{u})$ the right continuous pseudo-inverse of $(A_{t+\overline{g}})$.
$$\tau_{u}\equiv inf\{t\geq0; A_{t+\overline{g}}>u\}.$$
Let $\varphi:\R_{+}\rightarrow\R_{+}$ be a Borel function. Then, we have the following estimates:
\begin{equation}
\Pv^{'}(\exists t\geq\overline{g}, X_{t}>\varphi(A_{t})|\mathcal{F}_{\overline{g}}^{g})=M_{\overline{g}}\wedge1
\end{equation}
and
\begin{equation}
\Pv^{'}(\exists t \in [\overline{g},\overline{g}+\tau_{u}], X_{t}>\varphi(A_{t})|\mathcal{F}_{\overline{g}}^{g})=M^{u}_{\overline{g}}\wedge1    
\end{equation}
\end{thm}
\begin{proof}
Note that we can always assume that $\frac{1}{\varphi}$ is bounded and integrable (see proof of theorem 3.2 of A.Nikeghbali \cite{nik}). Hence from the theorem \ref{t1}, we get that $M$ is a nonnegative, $(\Qv,\Pv)-$ local martingale. Therefore from remark \ref{mart}, $(\widetilde{M}_{t})$ is  a nonnegative $\Pv^{'}-$ local martingale (whose supremum is continuous since $M$ has only negative jumps). Moreover, $(\widetilde{M}_{t})$ converges almost surely as $t\rightarrow\infty$ because it is a nonnegative martingale. Let us now consider $(\widetilde{M}_{\tau_{u}})$:
$$\widetilde{M}_{\tau_{u}}=F(u)-f(u)\widetilde{X}_{\tau_{u}}.$$
But since $(d\widetilde{A}_{t})$ is carried by the zeros of $(\widetilde{X}_{t})$ and since $\tau_{u}$ corresponds to an increase time of $\widetilde{A}_{t}$, we have 
$$\widetilde{X}_{\tau_{u}}=0.$$
Consequently,
$$\lim_{u\rightarrow\infty}{\widetilde{M}_{\tau_{u}}}=\lim_{u\rightarrow\infty}{F(u)}=0.$$
And hence
$$\lim_{u\rightarrow\infty}{M_{u}}=\lim_{u\rightarrow\infty}{\widetilde{M}_{u}}=0.$$
Now let us note that if for a given $t_{0}<\infty$, we have $\widetilde{X}_{t_{0}}>\varphi(\widetilde{A}_{t_{0}})$, then we must have
$$\widetilde{M}_{t_{0}}>F(\widetilde{A}_{t_{0}})-f(\widetilde{A}_{t_{0}})\varphi(\widetilde{A}_{t_{0}})=1$$
and we deduce that
$$\Pv^{'}(\exists t\geq\overline{g}, X_{t}>\varphi(A_{t})|\mathcal{F}_{\overline{g}}^{g})=\Pv^{'}\left(\sup_{t\geq\overline{g}}{M_{t}}>1|\mathcal{F}_{\overline{g}}^{g}\right)$$
$$\Pv^{'}(\exists t\geq\overline{g},X_{t}>\varphi(A_{t})|\mathcal{F}_{\overline{g}}^{g})=\Pv^{'}\left({\sup_{t\geq\overline{g}}{\frac{M_{t}}{M_{\overline{g}}}}>\frac{1}{M_{\overline{g}}}}|\mathcal{F}_{\overline{g}}^{g}\right)$$
Therefore from lemma \ref{l1}, we get:
$$\Pv^{'}(\exists t\geq\overline{g}, X_{t}>\varphi(A_{t})|\mathcal{F}_{\overline{g}}^{g})=M_{\overline{g}}\wedge1.$$
As in proof of theorem 3.2 of A.Nikeghbali \cite{nik}, it is enough to replace $\varphi$ by the function $\varphi_{u}$ defined as
$$\varphi_{u}(x)= \left \lbrace \begin{array}{l}
                      \varphi(x)$ $if$ $x<u\\
                      \infty$ $otherwise
                  \end{array} \right.$$
to obtain the second identity of the theorem.
\end{proof}
\begin{coro}
If in addition to the assumptions of the previous theorem, we have: $H\subset \{t: X_{t}=0\}$ and $A_{\overline{g}}=0$. Then,
\begin{equation}
\Pv^{'}(\exists t\geq\overline{g}, X_{t}>\varphi(A_{t}))=1-\exp\left(-\int_{0}^{+\infty}{\frac{dz}{\varphi(z)}}\right)
\end{equation}
and
\begin{equation}
\Pv^{'}(\exists t \in [\overline{g},\overline{g}+\tau_{u}], X_{t}>\varphi(A_{t}))=1-\exp\left(-\int_{0}^{u}{\frac{dz}{\varphi(z)}}\right)   
\end{equation}
\end{coro}
\begin{proof}
It is enough to see that under these assumptions, we have:
$$M_{\overline{g}}=1-\exp\left(-\int_{0}^{+\infty}{\frac{dz}{\varphi(z)}}\right)$$
and
$$M^{u}_{\overline{g}}=1-\exp\left(-\int_{0}^{u}{\frac{dz}{\varphi(z)}}\right).$$
Note that we could also use the corollary \ref{c0} for easily demonstrate this result.

\end{proof}

Now, we shall extend the corollary 3.6 of A.Nikeghbali \cite{nik}.
\begin{coro}
Let $X$ be an adapted process with respect to the filtration $(\mathcal{F}_{t})_{t\geq0}$ which vanishes at zero. Assume that $(X_{t})_{t\geq0}$ and $(X^{2}_{t}-t)_{t\geq0}$ are $(\Qv,\Pv)-$ local martingales. Let $(S_{t})_{t\geq0}$ denote the supremum process of $X$. Then, for all nonnegative Borel function  $\varphi$, we have:
\begin{equation}
\Qv(\forall t\geq0, S_{t}-X_{t}\leq\varphi(S_{t}))=\Qv(1)\exp{\left(-\int_{0}^{+\infty}{\frac{dx}{\varphi(x)}}\right)}
\end{equation}
Furthermore, if we let $T_{x}$ denote the stopping time:
$$T_{x}=inf\{t\geq0; S_{t}>x\}=inf\{t\geq0; X_{t}>x\}$$
then for any nonnegative Borel function $\varphi$, we have:
\begin{equation}
\Qv(\forall t\geq T_{x}, S_{t}-X_{t}\leq\varphi(S_{t}))=\Qv(1)\exp{\left(-\int_{0}^{x}{\frac{du}{\varphi(u)}}\right)}
\end{equation}
\end{coro}
\begin{proof}
Let us note 
$$G=\sigma(X_{t}, t\geq0).$$
According to theorem 1 of J.R.Chavez \cite{chav}, we have:\\
if $\Qv(1)=0$, hence $\Qv|_{G}=0$. In this case, the equalities of the corollary are verified. Otherwise, $\frac{1}{\Qv(1)}\Qv|_{G}$ is a probability measure under which $(X_{t})_{t\geq0}$ is a Brownian motion. Then, $(S_{t}-X_{t})_{t\geq0}$ is also of class $\left(\sum\right)$ with respect to  $\frac{1}{\Qv(1)}\Qv|_{G}$.
Consequently from the corollary 3.6 of A.Nikeghbali \cite{nik}, we get:
$$\frac{1}{\Qv(1)}\Qv(\forall t\geq0, S_{t}-X_{t}\leq\varphi(S_{t}))=\exp{\left(-\int_{0}^{+\infty}{\frac{dx}{\varphi(x)}}\right)}$$
and
$$\frac{1}{\Qv(1)}\Qv(\forall t\geq T_{x}, S_{t}-X_{t}\leq\varphi(S_{t}))=\exp{\left(-\int_{0}^{x}{\frac{du}{\varphi(u)}}\right)}.$$
This completes the proof.
\end{proof}

Now, we shall study the distribution of $A_{\infty}$
\begin{thm}
Let $X$ be a process of class $\left(\sum(H)\right)$ and of class $(\Dv)$ such that $H\subset \{t: X_{t}=0\}$ and $A_{\overline{g}}=0$. And define
$$\lambda(x)=\E^{'}[X_{\infty}|A_{\infty}=x].$$
$\E^{'}$ is the expectation with respect to $\Pv^{'}$.
Assume that $\lambda(A_{\infty})\neq0$. Then, if we note. 
$$b\equiv inf\{u: \Pv^{'}(A_{\infty}\geq u)=0\},$$ 
we have:
$$\Pv^{'}(A_{\infty}>x)=\exp\left(-\int_{0}^{x}\frac{dz}{\lambda(z)}\right), x<b.$$
\end{thm}
\begin{proof}
Let $f$ be a bounded Borel function with compact support; from corollary \ref{r1},
$$M_{t}=F(A_{t})-f(A_{t})X_{t}$$
is uniformly integrable $(\Qv,\Pv)-$ martingale. Hence, $(\widetilde{M}_{t}=M_{t+\overline{g}})_{t\geq0}$ is a \\$\Pv^{'}-$martingale. Moreover $\widetilde{M}_{0}=0$ (see remark ). Then
\begin{equation}\label{e}
\E^{'}[F(A_{\infty})]=\E^{'}[X_{\infty}f(A_{\infty})].
\end{equation}
Now, since $\lambda(A_{\infty})>0$, if $\nu(dx)$ denote the law of $A_{\infty}$, and $\overline{\nu}(x)=\nu([x,\infty[), \eqref{e}$ implies:
$$\int_{0}^{\infty}{dzf(z)\overline{\nu}(z)}=\int_{0}^{\infty}{\nu(dz)f(z)\lambda(z)}$$
and consequently,
\begin{equation}\label{e1}
\overline{\nu}(z)dz=\lambda(z)\nu(dz).
\end{equation}
Recall that $b\equiv inf\{u: \Pv^{'}(A_{\infty}\geq u)=0\}$; hence for $x<b$,
$$\int_{0}^{x}{\frac{dz}{\lambda(z)}}=\int_{0}^{x}{\frac{\nu(dz)}{\overline{\nu}(z)}}\leq\frac{1}{x}<\infty$$
and integrating \eqref{e1} between 0 and $x$, for $x<b$ yields
$$\overline{\nu}(x)=\exp\left(-\int_{0}^{x}{\frac{dz}{\lambda(z)}}\right),$$
and the result of the theorem follows easily.
\end{proof}

Note that a martingale with respect to a signed measure $\Qv$ in the meaning of J.R.Chavez, is necessarily a $\Pv-$semimartingale. It is for this unique reason that we can apply stochastic calculus on this family of processes. But, a martingale with respect to a signed measure $\Qv$ defined by S.Beghdadi Sakrani, is not necessarily a $\Pv-$semimartingale. Hence, we can not do stochastic calculus on these processes. However, S.Beghdadi Sakrani in \cite{sak} has established a theory of stochastic calculus for signed measures taking into account our famous processes. Therefore, we shall establish an extension of class $\big(\sum\big)$ in this theory.

\section{Presentation of second class of processes.}

In this section, we shall take $\Pv =|\Qv|$, where $\Qv$ is a signed measure such that $|\Qv|(\Omega)=1$. It will be necessary to recall some properties of stochastic calculus for signed measures before  defining the new class of processes.
\subsection{Some general properties of stochastic calculus for signed measures.}

We shall begin by give the following definitions.
\begin{defn} 
Let $X$ be an adapted process with respect to the filtration $\mathcal{F}$.
\begin{enumerate}
	\item $X$ is called a $\Qv-$martingale if: $\E[|X_{t}|]<+\infty$, $\forall t\geq0$ and $\Qv(X_{T})=\Qv(X_{0})$ for any bounded stopping time $T$.
	\item $X$ is called uniformly integrable $\Qv-$martingale if $XD$ is  a $\Pv$-martingale which is uniformly integrable .
	\item $X$ is a $\Qv$-local martingale if $DX$ is a $\Pv$-local martingale.
\end{enumerate} 

\end{defn}
\begin{defn}\hfill
\begin{enumerate}
	\item An adapted process $ (X _ {t}) _ {t\geq0} $ is said non-decreasing (resp. of finite variation) with respect to $\Qv$ if the process       $(\widetilde{X}_{t})_{t\geq0}=(X_{\overline{g}+t})_{t\geq0}$ is non-decreasing (resp. of finite variation) with respect to $\Pv$.
	\item  $A$ process $Y$ is said $\Qv-$semimartingale if $X = Y + A$, where $X$ is a uniformly integrable $\Qv-$martingale and $A$ is a process of finite variation with respect to $\Qv$. 
\end{enumerate}
\end{defn}

Now, we recall some properties of stochastic calculus for signed measures. we begin by quoting the following proposition of J. Azema and M. Yor $\cite{1}$.
\begin{prop}\label{14}
Let $(V_{t})_{t\geq0}$ be a $(\mathcal{F}_{g+t})_{t\geq0}$-optional process. There exists a unique $(\mathcal{F}_{t})_{t\geq0}$-optional process $(U_{t})_{t\geq0}$  which vanishes on $H$ such that $\forall t\geq0$, $U_{\overline{g}+t}=V_{t}$ and $U_{0}=V_{0}$ on $\{\overline{g}=0\}$. That defines a function,
$\rho:V\longmapsto U$.\\
$\rho$ is linear, nonnegative and preserves products.
\end{prop}
\begin{proof}
(See J.Azema and M.Yor \cite{1})
\end{proof}

The following definition which appears in S. Beghdadi-Sakrani  \cite{sak} defines the stochastic integral with respect to the signed measure $\Qv$.
\begin{defn}
Let $X$ be a uniformly integrable $\Qv$- martingale (resp. finite variation with respect to $\Qv$) and h, a progessive process such that:\\
$\int^{t}_{0}h^{2}_{\overline{g}+s}d\langle X_{\overline{g}+.}\rangle_{s}<+\infty$, $\forall t\geq0$ (resp. $\int^{t}_{0}|h_{\overline{g}+s}||d X_{\overline{g}+s}|<+\infty$). We define the stochastic integral of $h$ with respect to $X$ under the signed measure $\Qv $ by:
\begin{equation}
_\Qv\int^{t}_{0}h_{s}dX_{s}=\rho(\int_{0}^{.}h_{\overline{g}+s}d\widetilde{X}_{s})_{t}.
\end{equation}
\end{defn}

We note $[X]^{\Qv}=\rho([\widetilde{X}])_{.}$.
\begin{thm}
$[X]^{\Qv}$ is the unique process adapted, right continuous, non-decreasing on $[\overline{g}, +\infty[$ and null on $H$ such that $X^{2}-[X]^{\Qv}$ is a $\Qv$-local martingale.
\end{thm}
\begin{proof}
(See S.Beghdadi Sakrani \cite{sak})
\end{proof}

To end this subsection, we recall Itô theorem for signed measures. 
\begin{thm}(Itô Theorem)\label{l2}\\
Let $X:=(X^{1},\ldots,X^{d})$ be a vector of $d$ right continuous $\Qv$-semimartingales and $F\in$ $\mathcal{C}^{2}(\R^{d},\R)$. Then $F(X)$ is a right contuous $\Qv$-semimartingale. And for any $t\geq0$,
\begin{equation}
F(X_{t})=F(X_{\gamma_{t}})+\sum_{i}{_\Qv\int^{t}_{0}}\frac{\partial F}{\partial x_{i}}(X_{s})dX^{i}_{s}+\frac{1}{2}\sum_{i,j}{_\Qv\int^{t}_{0}}\frac{\partial^{2} F}{\partial x_{i}\partial x_{j}}(X_{s})d[X^{i},X^{j}]^{\Qv}_{s}
\end{equation}
\end{thm}
\begin{proof}
(See S.Beghdadi Sakrani \cite{sak})
\end{proof}

\subsection{Definition of second class of processes.}
\begin{defn}
Let $(X_{t})_{t\geq0}$ be a nonnegative process, which decomposes as:
$$X_{t}=N_{t}+A_{t}$$
We say that $(X_{t})_{t\geq0}$ is of class $\left(\sum_{s}(H)\right)$ if:
\begin{enumerate}
	\item $(N_{t})_{t\geq0}$ is a cadlag, uniformly integrable $\Qv$-martingale and null on $H$
	\item $(A_{t})_{t\geq0}$ is a process which is continuous on $]\overline{g},\infty[$ and null on $H$ such that: $\widetilde{A}_{t}=A_{\overline{g}+t}$ is non-decreasing
	\item the measure $(d\widetilde{A}_{t})$ is carried by the set $\{t: \widetilde{X}_{t}=0\}$.
\end{enumerate}
\end{defn}

 The following proposition provides the link that exists between the class $(\sum)$ and the class $\left(\sum_{s}(H)\right)$.
 
 \begin{prop}\label{p1}
 If $(X_{t})_{t\geq0}$ is of class $\left(\sum_{s}(H)\right)$, then $(\widetilde{X}_{t})_{t\geq0}$ is of class $(\sum)$.
 \end{prop}
 \begin{proof}
 Suppose that $X=N+A$ is of class $\left(\sum_{s}(H)\right)$. Hence, $N$ is a uniformly integrable $\Qv$-martingale. Then, according to the remark \ref{mart},  $\widetilde{N}$ is a uniformly integrable $\Pv$-martingale. But by hypothesis, $\widetilde{A}$ is a continuous non-decreasing  process, with $\widetilde{A}_{0}=0$ and $(d\widetilde{A}_{t})$ is carried by the set $\{t: \widetilde{X}_{t}=0\}$.\\
 Then, $(\widetilde{X}_{t})_{t\geq0}$ is of class $(\sum)$.
 \end{proof}
 
 The following theorem characterises the processes of class $\left(\sum_{s}(H)\right)$.
 
 \begin{thm}
 The following are equivalent:
\begin{enumerate}
	\item $X=N+A$ is of class $\left(\sum_{s}(H)\right)$;
	\item There exists an adapted process $(C_{t})_{t\geq0}$ which is continuous on $]\overline{g},\infty[$, null on $H$ and non-decreasing with respect to $\Qv$ such that for every locally bounded Borel function $f$, the process
	$$_{\Qv}\int^{t}_{0}f(C_{s})dC_{s}-\rho(f(C_{\overline{g}+.}))_{t}X_{t}$$
is a uniformly integrable $\Qv$-martingale. Moreover, $(C_{t})_{t\geq0}$ is equal to $(A_{t})_{t\geq0}$.
\end{enumerate}
 \end{thm}
 \begin{proof}
 $(1)\Rightarrow(2)$\\
 $X$ is of class $\left(\sum_{s}(H)\right)\Rightarrow \widetilde{X}$ is of class $(\sum)$. According to the theorem 2.1 of A.Nikeghbali \cite{nik}, for every locally bounded Borel function $f$, we have:
 $$\int_{0}^{t}f(A_{\overline{g}+s})dA_{\overline{g}+s}-f(A_{\overline{g}+t})X_{\overline{g}+t}$$ 
is a uniformly integrable $\Pv$-martingale with respect to $(\mathcal{F}_{g+t})_{t\geq0}$.\\
Hence from the quotient theorem of J.Azema and M.Yor \cite{1},
$$\rho\big(\int_{0}^{.}f(A_{\overline{g}+s})dA_{\overline{g}+s}-f(A_{\overline{g}+.})X_{\overline{g}+.}\big)_{t}$$
is a uniformly integrable $\Qv$-martingale with respect to $(\mathcal{F}_{t})_{t\geq0}$.\\
Then,
$$_{\Qv}\int^{t}_{0}f(A_{s})dA_{s}-\rho(f(A_{\overline{g}+.}))_{t}\rho(X_{\overline{g}+.})_{t}$$
is a uniformly integrable $\Qv$-martingale. But $X$ is null on $H$.
Therefore we get,
$$_{\Qv}\int^{t}_{0}f(A_{s})dA_{s}-\rho(f(A_{\overline{g}+.}))_{t}X_{t}$$
is a uniformly integrable $\Qv$-martingale.\\
$(2)\Rightarrow(1)$\\
Suppose that There exists an adapted process $(C_{t})_{t\geq0}$ which is continuous on $]\overline{g},\infty[$, null on $H$ and non-decreasing with respect to $\Qv$ such that for every locally bounded Borel function $f$, the process
	$$_{\Qv}\int^{t}_{0}f(C_{s})dC_{s}-\rho(f(C_{\overline{g}+.}))_{t}X_{t}$$
is a uniformly integrable $\Qv$-martingale.\\
Then we obtain that,
 $$\int_{0}^{t}f(C_{\overline{g}+s})dC_{\overline{g}+s}-f(C_{\overline{g}+t})X_{\overline{g}+t}$$ 
is a uniformly integrable $\Pv$-martingale with respect to $(\mathcal{F}_{g+t})_{t\geq0}$.\\
Then, by theorem 2.1 of A.Nikeghbali \cite{nik}, $(X_{\overline{g}+t})$ is of class $(\sum)$ and $C_{\overline{g}+t}=A_{\overline{g}+t}$, $\forall t\geq0$.\\
Hence,
$$
\left \lbrace
\begin{array}{l}
X_{\overline{g}+t}-A_{\overline{g}+t}$ $is$ $a$ $uniformly$ $intingrable$ $local$ $martingale\\
C_{\overline{g}+t}=A_{\overline{g}+t}\\
(dA_{\overline{g}+t})$ $is$ $carried$ $by$ $\{t; X_{\overline{g}+t}=0\}
\end{array} \right.
$$
that implies,
$$
\left \lbrace
\begin{array}{l}
\rho\big(X_{\overline{g}+.}-A_{\overline{g}+.}\big)_{t}$ $is$ $a$ $uniformly$ $intingrable$ $\Qv-local$ $martingale\\
\rho\big(C_{\overline{g}+.}\big)_{t}=\rho\big(A_{\overline{g}+.}\big)_{t}\\
(dA_{\overline{g}+t})$ $is$ $carried$ $by$ $\{t; X_{\overline{g}+t}=0\}
\end{array} \right.
$$
Therefore,
$$
\left \lbrace
\begin{array}{l}
X_{t}-A_{t}$ $is$ $a$ $uniformly$ $intingrable$ $\Qv-local$ $martingale\\
C_{t}=A_{t}\\
(dA_{\overline{g}+t})$ $is$ $carried$ $by$ $\{t; X_{\overline{g}+t}=0\}
\end{array} \right.
$$
Consequently, $X$ is of class $\left(\sum_{s}(H)\right)$.
 \end{proof}
\begin{rem}
When f is nonnegative with $f(0)=0$, this means that
$(f (A_{t})X_{t})_{t\geq0}$ is of class $\left(\sum_{s}(H)\right)$ and its non-decreasing process under the signed measure $\Qv$ is 
$$G(A_{t})={_{\Qv}\int_{0}^{t}f(A_{s})dA_{s}}$$.
\end{rem}
indeed, by applying the theorem 2.2 of S.Beghdadi Sakrani \cite{sak}, we get:
$$f(A_{t})X_{t}={_{\Qv}\int_{0}^{t}f(A_{s})dN_{s}}+{_{\Qv}\int_{0}^{t}f(A_{s})dA_{s}}.$$
But, $M_{t}={_{\Qv}\int_{0}^{t}f(A_{s})dN_{s}}$ is a uniformly integrable $\Qv$-martingale beacause, $N$ is too (see proposition 2.3 of sak). Moreover,
 $$G(A_{\overline{g}+t})=\int_{0}^{t}f(A_{\overline{g}+s})dA_{\overline{g}+s}$$
 is a non-decreasing process, $(dG(A_{\overline{g}+t}))$ is a measure carried by $\{t; f(A_{\overline{g}+t})X_{\overline{g}+t}=0\}$, the processes $M$ and $G(A_{.})$ are null on $H$. Therefore, $(f (A_{t})X_{t})_{t\geq0}$ is of class $\left(\sum_{s}(H)\right)$.
 
The next proposition shows that the product of processes of class $\left(\sum_{s}(H)\right)$ with vanishing quadratic covariations with respect to signed measure
is again of class $\left(\sum_{s}(H)\right)$.

\begin{prop}
Let $(X_{t}^{1}),\ldots,(X_{t}^{n})$ be processes of class $\left(\sum_{s}(H)\right)$ such that \\
$\langle X^{i},X^{j} \rangle_{t}^{\Qv}=0$ for $i\neq j.$ Then $\Pi^{n}_{i=1}X^{i}_{t}$ is again of class $\left(\sum_{s}(H)\right)$.
\end{prop}
\begin{proof}
Since $\langle X^{1},X^{2} \rangle_{t}^{\Qv}=0$, the theorem  2.2 of S.Beghdadi Sakrani \cite{sak} yields
$$X^{1}_{t}X^{2}_{t}=_{\Qv}\int^{t}_{0}X_{s}^{2}dX_{s}^{1}+_{\Qv}\int^{t}_{0}X_{s}^{1}dX_{s}^{2}.$$
Hence,
$$X^{1}_{t}X^{2}_{t}=_{\Qv}\int^{t}_{0}X_{s}^{2}dN_{s}^{1}+_{\Qv}\int^{t}_{0}X_{s}^{1}dN_{s}^{2}+_{\Qv}\int^{t}_{0}X_{s}^{2}dA_{s}^{1}+_{\Qv}\int^{t}_{0}X_{s}^{1}dA_{s}^{2}$$
$N^{'}_{t}=_{\Qv}\int^{t}_{0}X_{s}^{2}dN_{s}^{1}+_{\Qv}\int^{t}_{0}X_{s}^{1}dN_{s}^{2}$ is a uniformly integrable $\Qv$-martingale null on $H$,    $A^{'}_{t}=_{\Qv}\int^{t}_{0}X_{s}^{2}dA_{s}^{1}+_{\Qv}\int^{t}_{0}X_{s}^{1}dA_{s}^{2}$ is a continuous process null on $H$ and non-decreasing with respect to $\Qv$ and $(dA^{'}_{.+\overline{g}})$ is carried by $\{t: X^{1}_{t+\overline{g}}X^{2}_{t+\overline{g}}=0\}$. Then, $(X_{t}^{1}X_{t}^{2})_{t\geq0}$ is of class $\left(\sum_{s}(H)\right)$. If $n\geq3$, then $\langle X^{1}X^{2},X^{3} \rangle_{t}^{\Qv}=0$ and the proposition follows by induction.
\end{proof}

The following theorem gives the Tanaka formulas for signed measures. It would be useful for the examples.
\begin{thm}
Let X be a right continuous $\Qv-$semimartingal. For all real $a$, we have:
\begin{equation} \label{eq1}
|X_{t}-a|=|X_{\gamma_{t}}-a|+_{\Qv}\int^{t}_{0}sgn(X_{s}-a)dX_{s}+_{\Qv}L_{t}^{a}(X)
\end{equation}

\begin{equation} \label{eq2}
(X_{t}-a)^{+}=(X_{\gamma_{t}}-a)^{+}+_{\Qv}\int^{t}_{0}1_{\{X_{s}>a\}}dX_{s}+\frac{1}{2}{_{\Qv}L_{t}^{a}(X)}
\end{equation}

\begin{equation} \label{eq3}
(X_{t}-a)^{-}=(X_{\gamma_{t}}-a)^{-}-_{\Qv}\int^{t}_{0}1_{\{X_{s}\leq a\}}dX_{s}+\frac{1}{2}{_{\Qv}L_{t}^{a}(X)}
\end{equation}
where ${_{\Qv}L_{t}^{a}(X)}=\rho \big(L_{.}^{a}(X_{.+\overline{g}})\big)_{t}$ and $(L_{t}^{a}(X_{.+\overline{g}}))_{t\geq0}$ is the classical semimartingale local time of $(X_{t+\overline{g}})_{t\geq0}$.
\end{thm}
\begin{proof}\\
Since $X$ is a continuous and $(\mathcal{F}_{t})_{t\geq0}$-adapted process, hence\\ $(V_{t}=|X_{.+\overline{g}}-a|-|X_{\overline{g}}-a|)_{t\geq0}$ is a $(\mathcal{F}_{t+g})_{t\geq0}$-optional process. Then from the proposition \ref{14}, there exists a unique $(\mathcal{F}_{t})_{t\geq0}$-optional process $(U_{t})_{t\geq0}$, null on $H$ such that $\forall t>0$, $V_{t+\overline{g}}=U_{t+\overline{g}}$ and $V_{0}=U_{0}$ on $\{\overline{g}=0\}$.\\
Since $(|X_{t}-a|-|X_{\gamma_{t}}-a|)_{t\geq0}$ is a $(\mathcal{F}_{t})_{t\geq0}$-optional process and
$$(|X_{.}-a|-|X_{\gamma_{.}}-a|)_{t+\overline{g}}=|X_{t+\overline{g}}-a|-|X_{\gamma_{t+\overline{g}}}-a|$$
$$\hspace{4.25cm}=|X_{t+\overline{g}}-a|-|X_{\overline{g}}-a|$$
Because $\forall t\geq0$, $\gamma_{t+\overline{g}}=\overline{g}$.\\
Hence,
$$U_{.}=|X_{.}-a|-|X_{\gamma_{.}}-a|$$
Then we have:
$$|X_{t}-a|=|X_{\gamma_{t}}-a|+\rho(|X_{.+\overline{g}}-a|-|X_{\overline{g}}-a|)_{t}$$
but $X_{t+\overline{g}}$ is a $\Pv-$semimartingal. Then, we can apply the usual Tanaka formula,\\
and we obtain,\\
$|X_{t}-a|=|X_{\gamma_{t}}-a|+\rho(\int^{.}_{0}sgn(X_{s+\overline{g}}-a)dX_{s+\overline{g}}+L_{t}^{a}(X_{.+\overline{g}}))_{t}$, and\\
$|X_{t}-a|=|X_{\gamma_{t}}-a|+\rho(\int^{.}_{0}sgn(X_{s+\overline{g}}-a)dX_{s+\overline{g}}+L_{.}^{a}(X_{.+\overline{g}}))_{t}$.\\
Therefore, we have:\\
$|X_{t}-a|=|X_{\gamma_{t}}-a|+_{\Qv}\int^{t}_{0}sgn(X_{s}-a)dX_{s}+_{\Qv}L_{t}^{a}(X)$\\
We proceed in the same way to demonstrate the points $\eqref{eq2}$ and $\eqref{eq3}$.
\end{proof}

Let us recall the class $\mathcal{R}_{+}$ introduced by J.Azema and M.Yor \cite{1}.
\begin{defn}
Let $(Y_{t})$ be a nonnegative process of the form $Y_{t}=N_{t}+A_{t}$. We say that $(Y_{t})$ is of class $\mathcal{R}_{+}$ if:
\begin{enumerate}
	\item the set $\{t: Y_{t}=0\}$ is closed;
	\item $N$ is uniformly integrable martingale which is right continuous;
	\item $A$ is a continuous process which is non-decreasing and integrable such that $(dA_{t})$ is carried by $\{t: Y_{t}=0\}$;
	\item $\Pv(Y_{\infty}=0)=0$.
\end{enumerate}
\end{defn}

Now, we shall give some examples of processes in the class $\big(\sum_{s}(H)\big)$.
\begin{exple}
\begin{enumerate}
	\item Let $Y=N+A$ be an adapted process of class $\mathcal{R}_{+}$ with respect to the filtration $(\mathcal{F}_{t}^{g})$ such that $H\subset \{t: Y_{t}=0\}$ and $N_{\overline{g}}=0$. Then $Y$ is also of class $\big(\sum_{s}(H)\big)$ with respect to the filtration $(\mathcal{F}_{t})$.
	
 Indeed $N_{t}=Y_{t}-A_{t}$ is a uniformly integrable martingale. Since $\overline{g}$ is a stopping time in $(\mathcal{F}_{t}^{g})$, hence $N_{t+\overline{g}}=Y_{t+\overline{g}}-A_{t+\overline{g}}$ is a uniformly integrable martingale with respect to $(\mathcal{F}_{t+g})$. Then by applying the quotient theorem, we get that $\rho(N_{\overline{g}+.})_{t}$ is a uniformly integrable $\Qv-$martingale with respect to the filtration $(\mathcal{F}_{t})$. Therefore $\rho(Y_{\overline{g}+.})_{t}=\rho(N_{\overline{g}+.})_{t}+\rho(A_{\overline{g}+.})_{t}$ is of class $\big(\sum_{s}(H)\big)$. But $Y_{t}=\rho(Y_{\overline{g}+.})_{t}$ since $H\subset \{t: Y_{t}=0\}$. This completes the demonstration.
	\item Let $M$ be a continuous, uniformly integrable $\Qv-$martingale. Then,
	$$|M_{t}|={_{\Qv}\int_{0}^{t}sgn(M_{s})dM_{s}+_{\Qv}L_{t}^{0}(M)}$$
	is of class $\big(\sum_{s}(H)\big)$.
	\item Similarly, for any $\alpha>0$ and $\beta>0$, the process
	$$\alpha M_{t}^{+}+\beta M_{t}^{-}$$
	is of class $\big(\sum_{s}(H)\big)$.
	\item Let $M$ be a right continuous, uniformly integrable $\Qv-$martingale vanishing on $H$ such that $(M_{\overline{g}+t})$ is with only negative jumps. Let $S$ be the supremum process of $(M_{\overline{g}+t})$. Then
	$$X_{t}=\rho(S)_{t}-M_{t}$$
	is of class $\big(\sum_{s}(H)\big)$.
\end{enumerate}
\end{exple}

Now, we shall give some estimates and distributions for the pair $(X_{t},A_{t})$.
\begin{thm}
Let $X$ be a process of the class $\big(\sum_{s}(H)\big)$, with only negative jumps, such that $A_{\infty}=\infty$. Define $(\tau_{u})$ the right continuous inverse of $(A_{t+\overline{g}})$. 
$$\tau_{u}\equiv inf\{t\geq0; A_{t+\overline{g}}>u\}.$$
Let $\varphi:\R_{+}\rightarrow\R_{+}$ be a Borel function. Then, we have the following estimates:
\begin{equation}
\Pv(\exists t\geq\overline{g}, X_{t}>\varphi(A_{t}))=1-\exp\left(-\int_{0}^{+\infty}{\frac{dz}{\varphi(z)}}\right)
\end{equation}
and
\begin{equation}
\Pv(\exists t \in [\overline{g},\tau_{u}+\overline{g}], X_{t}>\varphi(A_{t}))=1-\exp\left(-\int_{0}^{u}{\frac{dz}{\varphi(z)}}\right).   
\end{equation}
\end{thm}

\begin{proof}
From the proposition \ref{p1}, $\widetilde{X}$ is of class $\big(\sum\big)$. Since by assumptions, $\widetilde{X}$ is also a process with negative jumps such that $\widetilde{A}_{+\infty}=+\infty$. Then applying the theorem (3.2) of A.Nikeghbali \cite{nik}, we get
$$\Pv(\exists t\geq0, X_{t+\overline{g}}>\varphi(A_{t+\overline{g}}))=1-\exp\left(-\int_{0}^{+\infty}{\frac{dz}{\varphi(z)}}\right)$$
and
$$\Pv(\exists t\leq\tau_{u}, X_{t+\overline{g}}>\varphi(A_{t+\overline{g}}))=1-\exp\left(-\int_{0}^{u}{\frac{dz}{\varphi(z)}}\right).$$
Therefore
$$\Pv(\exists t\geq\overline{g}, X_{t}>\varphi(A_{t}))=1-\exp\left(-\int_{0}^{+\infty}{\frac{dz}{\varphi(z)}}\right)$$
and
$$\Pv(\exists t \in [\overline{g},\tau_{u}+\overline{g}], X_{t}>\varphi(A_{t}))=1-\exp\left(-\int_{0}^{u}{\frac{dz}{\varphi(z)}}\right).$$ 
\end{proof}



\begin{thebibliography}{1}

\bibitem{1}
J.Azéma and M.Yor.
\emph{Sur les zéros des martingales continues. Séminaire de probabilités (Strasbourg), 26: 248-306, 1992.}
\bibitem{sak}
S.Beghdadi-Sakrani.
\emph{Calcul stochastique pour les mesures signées. Séminaire de probabilités (Strasbourg), 36: 366-382, 2002.}

\bibitem{chav}
J.R.Chavez.
\emph{Le théorème de Paul Lévy pour des mesures signées. Séminaire de probabilités (Strasbourg), 18: 245-255, 1984.}

\bibitem{jeulin1980semi}
T.Jeulin.
\emph{Semi-martingales et grossissement d'une filtration. Lecture notes in mathematics.  Springer, 1980.}

\bibitem{JEUYOR}
T.Jeulin and M.Yor.
\emph{Grossissement d'une filtration et semi-martingales: formules explicites. Séminaire de Probabilités (Strasbourg), 12: 78-97, 1978.}
\bibitem{nik}
A.Nikeghbali,
\emph{A class of remarkable submartingales. Stochastic Processes and their Applications, 116: 917-938, 2006.}

\bibitem{max}
A.Nikeghbali and M.Yor.
\emph{Doob's maximal identity, multiplicative decompositions and enlargements of filtrations. Illinois J. Math.(in press), 2005.}


\bibitem{YOR1}
M.Yor.
\emph{Grossissement d'une filtration et semi-martingales:théorèmes généraux. Séminaire de probabilités (Strasbourg), 12: 61-69, 1978.}
\end{thebibliography}
\end{document}